\documentclass[11pt]{amsart}
\usepackage{amsmath}
\usepackage{amsthm}
\usepackage{mathdots}
\usepackage{stmaryrd}
\usepackage[initials, lite]{amsrefs}
\usepackage{color}
\usepackage{tabularx}
\usepackage{longtable}
\usepackage{booktabs}
\usepackage{multirow}
\usepackage{array}
\usepackage{amsfonts} 
\usepackage{paralist}
\usepackage{aliascnt}
\usepackage{amscd}
\usepackage{blkarray}
\usepackage{mathbbol}
\usepackage{tikz, tikz-cd}
\usepackage{amsmath,amscd,amsthm,amssymb,amsxtra,latexsym, epsfig,epic,graphics}
\usepackage{tikz}
\usetikzlibrary{matrix}
\usetikzlibrary{arrows,calc}
\allowdisplaybreaks

\BibSpec{collection.article}{%
	+{}  {\PrintAuthors}                {author}
	+{,} { \textit}                     {title}
	+{.} { }                            {part}
	+{:} { \textit}                     {subtitle}
	+{,} { \PrintContributions}         {contribution}
	+{,} { \PrintConference}            {conference}
	+{}  {\PrintBook}                   {book}
	+{,} { }                            {booktitle}
	+{,} { }                            {series}
	+{, vol.} { }                            {volume}
	+{,} { }                            {publisher}
	+{,} { \PrintDateB}                 {date}
	+{,} { pp.~}                        {pages}
	+{,} { }                            {status}
	+{,} { \PrintDOI}                   {doi}
	+{,} { available at \eprint}        {eprint}
	+{}  { \parenthesize}               {language}
	+{}  { \PrintTranslation}           {translation}
	+{;} { \PrintReprint}               {reprint}
	+{.} { }                            {note}
	+{.} {}                             {transition}
	+{}  {\SentenceSpace \PrintReviews} {review}
}
\AtBeginDocument{%
	\def\MR#1{}
}

\newcommand{\m}{\mathfrak{m}}

\newcommand{\reg}{\normalfont\text{reg}}

\newtheorem{theorem}{Theorem}[section]

\newaliascnt{headcor}{headthm}

\aliascntresetthe{headcor}

\newaliascnt{headconj}{headthm}

\aliascntresetthe{headconj}

\newaliascnt{corollary}{theorem}
\newtheorem{corollary}[corollary]{Corollary}
\aliascntresetthe{corollary}

\def\min{\operatorname{min}}

\def\reg{\operatorname{reg}}

\def\dim{\operatorname{dim}}

\def\htt{\operatorname{ht}}

\def\deg{\operatorname{deg}}

\def\m{\mathfrak m}

\newaliascnt{lemma}{theorem}

\aliascntresetthe{lemma}

\newaliascnt{conjecture}{theorem}

\aliascntresetthe{conjecture}

\newaliascnt{proposition}{theorem}
\newtheorem{proposition}[proposition]{Proposition}
\aliascntresetthe{proposition}

\theoremstyle{definition}
\newaliascnt{definition}{theorem}

\aliascntresetthe{definition}

\newaliascnt{notation}{theorem}

\aliascntresetthe{notation}

\newaliascnt{example}{theorem}
\newtheorem{example}[example]{Example}
\aliascntresetthe{example}

\newaliascnt{examples}{theorem}

\aliascntresetthe{examples}

\newaliascnt{remark}{theorem}
\newtheorem{remark}[remark]{Remark}
\aliascntresetthe{remark}

\newaliascnt{question}{theorem}

\aliascntresetthe{question}

\newaliascnt{questions}{theorem}

\aliascntresetthe{questions}

\newaliascnt{problem}{theorem}

\aliascntresetthe{problem}

\newaliascnt{construction}{theorem}

\aliascntresetthe{construction}

\newaliascnt{setting}{theorem}
\newtheorem{setting}[setting]{Setting}
\aliascntresetthe{setting}

\newaliascnt{algorithm}{theorem}

\aliascntresetthe{algorithm}

\newaliascnt{observation}{theorem}

\aliascntresetthe{observation}

\newaliascnt{defprop}{theorem}

\aliascntresetthe{defprop}

\DeclareFontFamily{OT1}{pzc}{}
\DeclareFontShape{OT1}{pzc}{m}{it}{<-> s * [1.100] pzcmi7t}{}
\DeclareMathAlphabet{\mathchanc}{OT1}{pzc}{m}{it}

\begin{document}

\title[On the number of generators of licci ideals]{On the number of generators of licci ideals}
%\date\today



\thanks{This material is based upon work supported by the National Science Foundation under Grant No. DMS-1928930 and by the Alfred P. Sloan Foundation under grant G-2021-16778, while the authors were in residence at the Simons Laufer Mathematical Sciences Institute (formerly MSRI) in Berkeley, California, during the Spring 2024 semester. The second and third authors were partially supported by NSF grants DMS-2502707 and DMS-2502706, respectively.}

\author{Craig Huneke}
\address{Craig Huneke, Department of  Mathematics, University of Virginia, Charlottesville, VA 22903} \email{huneke@virginia.edu}

\author{Claudia Polini}
\address{Claudia Polini, Department of Mathematics, University of Notre Dame, Notre Dame, Indiana 46556} \email{cpolini@nd.edu}

\author{Bernd Ulrich}
\address{Bernd Ulrich, Department of Mathematics, Purdue University, West Lafayette, Indiana 47907} \email{bulrich@purdue.edu}
\begin{abstract}  
We prove a conjecture on the minimal number of generators of licci zero-dimensional ideals that are either monomial or have small Loewy colength.
\end{abstract}
 \maketitle
%\tableofcontents
\section{Introduction}

Linkage (or liaison) is a classical method for classifying and studying ideals and subvarieties. Originating in the nineteenth century, the subject continues to be 
%actively 
investigated 
from many points of view \cites{KTY2013, Chong, RT19, RTY2020, ERT2020, SW21, RT2023, CVW2019, GNW2024, GNW2412, GNW25, CGNW25, CGNW2503, JRS, MMM, HPU26-1, HPU26-2, HPU26-4}. Two proper ideals $I$ and $J$ of a Gorenstein ring $R$ are said to be {\it linked} if there exists a complete intersection ideal $\mathfrak{a}$ such that
\[
J=\mathfrak{a}:I \qquad \text{and} \qquad I=\mathfrak{a}:J .
\]
By iterating this construction one obtains the {\it linkage class} of an unmixed ideal $I$, that is, the collection of all ideals that can be obtained from $I$ through a finite sequence of links. Of particular interest are the ideals that belong to the linkage class of a complete intersection; these are called {\it licci} ideals. In this paper we study the number of generators of licci zero-dimensional ideals which are either monomial or of small Loewy colength. In  \cite{HU07}, there is a simple algorithm to decide whether a zero-dimensional monomial ideal is licci. This is one of the few cases in the literature where one has not only a necessary, but also a sufficient condition for an ideal to be licci. 

We prove conjectures on the minimal number of generators of licci ideals made in \cite{HPU26-1}. 
The conjectures imply, for instance, that squarefree licci monomial ideals are of 
linear type, meaning their Rees algebra and symmetric algebra are isomorphic. Moreover this algebra
is Cohen-Macaulay \cite{HPU26-1}. 

Let $I$ be a homogeneous licci ideal of height $g$ in a polynomial ring $S$ over a 
field. In \cite{HU87} it was shown that the maximal last shift $D$ in a minimal
homogeneous free resolution of $I$ is `large' compared to the initial degree of $I,$
$$D > (g-1) \cdot {\rm indeg}(I).$$
Surprisingly, the last shift $D$ also appears to bound the minimal number 
of generators $\mu(I).$ In fact, \cite{HPU26-1}*{Conjecture 1.2} states that
$$ D \geq \mu(I) $$
or, equivalently, that the \emph{deviation}
$
d(I):=\mu(I)-\mathrm{ht}(I)
$
%the \emph{deviation} of $I$, 
is bounded above by the Castelnuovo-Mumford regularity $\, \reg(S/I);$
%, the Castelnuovo--Mumford regularity of $S/I;$ 
see \cite{HPU26-1}. 

Since the ring $S/I$ is Cohen-Macaulay, its Castelnuovo-Mumford regularity is one less than the generalized Loewy length. Thus, our conjecture allows for a natural generalization to the
local case that, for an $\m$-primary licci ideal in a regular local ring $(R,\m),$
simply states that \cite{HPU26-1}*{Conjecture 1.3}
%\begin{comment}
%Let $I$ be a homogeneous ideal in a polynomial ring over a field, and let $\mu(I)$ denote its minimal number of generators. In \cite{HPU26-1} we posed the conjecture that 
%$\mu(I)$ is bounded above by the maximal last shift in a minimal homogeneous free resolution of $I$. 
%Equivalently, we conjecture that
%$
%d(I)=\mu(I)-\mathrm{ht}(I),
%$
%the \emph{deviation} of $I$, is bounded above by $\reg(S/I)$, the 
%Castelnuovo--Mumford regularity of $S/I$. 
%Notice that licci ideals are perfect; hence the conjecture makes sense in any graded ring. 
%The conjecture can also be generalized to the local case. Indeed, if $\mathfrak m$ denotes the maximal ideal of the ring, then for an $\mathfrak m$-primary licci ideal $I$ the conjecture is equivalent to the statement
%\end{comment}
\[
\m^{d(I)} \not\subset I .
\]

In this paper, we give a proof of this conjecture for both $\m$-primary licci monomial ideals 
(Theorem~\ref{Cmmon}) and
licci ideals containing the cube of the maximal ideal of a regular local ring (Theorem~\ref{mcubed}). We also discuss consequences
of these results for some licci ideals which are not $\m$-primary (Corollary~\ref{Cmon} and Theorem~\ref{Cstrong}).

\vspace{.05cm}

\section{Monomial ideals}
For a polynomial ring $S=k[x_1, \ldots, x_n]$ over a field $k$ we always write $\m$ for its homogeneous maximal ideal $(x_1,\ldots, x_n).$ 
%A monomial in $S$ is an element of the form $x_1^{a_1} \cdots x_n^{a_n}.$ 
%A monomial ideal is an ideal generated by monomials. 
Every $\m$-primary monomial ideal $I$ can be written uniquely in a standard form
$$I=(x_1^{a_1}, \ldots, x_n^{a_n})+ I^{\#},$$
where $I^{\#}$ is generated by monomials that together with $\{x_1^{a_1}, \ldots, x_n^{a_n}\}$ generate $I$ minimally. 

According to \cite{HU07}*{Theorem 2.4}, if 
%$I$ is an $\m$-primary monomial ideal and 
$I^{\#}$ has height at least two, then $I_{\m}$ is not licci in $S_{\m},$ in particular $I$ is not licci. Hence if $I$ is a licci $\m$-primary monomial ideal, we can write $I$ in a standard form as
$$I=(x_1^{a_1}, \ldots, x_n^{a_n})+ x^BK,$$
where $x^B K=I^{\#},$ $\htt K\ge 2,$ and $x^B=x_1^{b_1} \cdots x_n^{b_n} \not=1.$ In \cite{HU07}*{Theorem 2.6}
it was also shown that $I$ is licci if (and only if) $I_{\m}$ is licci.
%If $0\not =K \not=S,$ then the monomial ideal $$I'=(x_1^{a_1-b_1}, \ldots, x_n^{a_n-b_n}) +K$$ is obtained from $I$ by a double link, in particular $I'$ is still a licci $\m$-primary monomial ideal \cite{HU07}*{Lemma 2.5}. 
%Let $$0\lto \bigoplus_j S(-d_{nj}) \lto \ldots \lto \bigoplus_j S(-d_{1j}) \lto S \lto I \lto 0$$
%be a minimal free resolution of $I$. 
%In this section we will show that if $I$ is a licci $\m$-primary monomial ideal then our conjecture on the number generators holds. 
%$\mu(I),$ the minimal number of generators of $I,$ is bounded above by $D_n={\rm max}\{d_{nj}\},$ or equivalently that $d(I)=\mu(I)-n,$ the {\it deviation} of $I,$ is bounded above by $\reg(S/I)$, the {\it Castelnuovo-Mumford regularity} of $S/I.$ This holds if and only if $\m^{d(I)} \not\subset I.$

\begin{setting}\label{S1}
Let $S=k[x_1, \ldots, x_n]$ be a polynomial ring over a field $k,$ and $I$ be a licci $\m$-primary monomial ideal.
%and assume that $I_{\m}$ is licci. 
Write
%According to \cite{HU07}*{Theorem 2.4}, we can write
\[
I=(x_1^{a_1}, \ldots, x_n^{a_n})+ x^B K,
\]
where $x^B K=I^{\#},$ $\htt K\ge 2,$ and $x^B=x_1^{b_1} \cdots x_n^{b_n} \not=1.$  Set
\[
A_i=\sum_{j\ne i}(a_j-1), 
\qquad 
N(I)=\min\{A_1,\ldots,A_n\},
\qquad
A(I)=\sum_{j=1}^n a_j.
\]
\end{setting}

\vspace{.2cm}

\begin{proposition}\label{TN}
With assumptions and notation as in Setting~\ref{S1} one has
\[
d(I)\le N(I).
\]
\end{proposition}

\begin{proof}
We argue by induction on $A(I)\ge n$. If $A(I)=n$, then $I$ is a complete intersection and hence $d(I)=0$, so the claim holds.

Assume now that the statement holds for every licci $\m$-primary monomial ideal $J$ with $A(J)<A(I)$, and let us prove it for $I$. If $I$ is a complete intersection there is nothing to prove, so we may suppose that $K\neq 0$ and $x^B\neq 1$. After a change of variables we may assume that
\[
b_1=\max\{b_j\}\ge 1 .
\]
As in the proof of \cite{HU07}*{Lemma 2.5}, the ideal
\[
I^*=(x_1^{a_1-1},x_2^{a_2}, \ldots, x_n^{a_n})+ \frac{x^B}{x_1}K
\]
is obtained from $I$ by a double link.
%defined by the monomial regular sequences
%\[
%x_1^{a_1}, \ldots, x_n^{a_n}
%\quad\text{and}\quad
%x_1^{a_1-1}, \ldots, x_n^{a_n}.
%\]
In particular, $I^*$ is again a licci $\m$-primary monomial ideal and $A(I^*)<A(I)$; hence the induction hypothesis applies to $I^*$. Thus
$
d(I^*)\le N(I^*),
$ and it suffices to show that 
\begin{equation}\label{EQ Difference} d(I)-d(I^*)\le N(I)-N(I^*).
\end{equation}
Notice that $N(I^*)\le N(I)$ and that no minimal generator of $\frac{x^B}{x_1}K$ is contained in $(x_1^{a_1-1},x_2^{a_2}, \ldots, x_n^{a_n}).$ 

First assume that $b_1\ge 2$ or $\deg x^B\ge 3$. We claim $d(I)=d(I^*),$ which will prove (\ref{EQ Difference}) in this case. It suffices to prove that none of the pure powers $x_1^{a_1-1},x_2^{a_2}, \ldots, x_n^{a_n}$ is contained in $\frac{x^B}{x_1}K.$  If $b_1\ge 2$, then every pure power $x_\ell^{c_\ell}$ in $K$ is multiplied by $x_1$, and $x_1^{a_1-1}\notin \frac{x^B}{x_1}K$ because $x_1^{a_1}\notin x^B K$. Otherwise, if $b_1=1$ and $\deg x^B\ge 3$, then every pure power $x_\ell^{c_\ell}$ in $K$ is multiplied by $x_ix_j$ with $2\le i<j\le n$, so %that $\mu(I)=\mu(I^*),$ thus $d(I)=d(I^*).$
no pure power is contained in $\frac{x^B}{x_1}K.$ Thus $d(I)=d(I^*)$ as claimed.

We may now assume that $b_1=\max\{b_j\}=1$ and $1\le \deg x^B\le 2$. If $\deg x^B=2,$ then $x^B=x_1x_j$ for some $2\le j\le n$. Now $d(I)=d(I^*)$ unless $x_j^{c_j}\in K$ with $c_j\le a_j-2.$ In the latter case $d(I^*)=d(I)-1, A_j(I^*)=A_j(I)-1$ and $A_i(I^*)\le A_i(I)-1$ for all $i\ne j$. Hence
\[
d(I)=d(I^*)+1
\quad\text{and}\quad
N(I^*)+1\le N(I),
\]
which implies (\ref{EQ Difference}) in this case. 

Finally, assume that $\deg x^B=1$, so that $x^B=x_1$. Let $t$ be the number of pure powers $x_\ell^{c_\ell}\in K$ such that $c_\ell<a_\ell$. Notice that $\ell\not=1,$ as $x_1^{a_1}\not \in x_1 K$. Passing from $I$ to $I^*$ we lose exactly $t$ minimal generators, so
\[
d(I)=d(I^*)+t.
\]
Moreover, since the power of $x_1$ decreases by 1 and the other $t$ pure powers decrease by at least 1,  each quantity $A_i$ decreases at least by $t$:
\[
A_i(I^*)\le A_i(I)-t
\qquad\text{for all } i.
\]
In particular
\[
N(I^*)+t\le N(I),
\]
again proving (\ref{EQ Difference}).
\end{proof}

% If $d(I)<d(I^*)$, then there exists a pure power $x_\ell^{c_\ell}\in K$ with $c_\ell<a_\ell$. Note that $2\le \ell\le n$, since $x_1^{a_1}\notin x_1K$. Let $t$ denote the number of such pure powers. After a change of variables we may assume that $\ell=2,\ldots,t+1$ and write
%\[
%I^*=(x_1^{a_1-1}, x_2^{c_2+1}, \ldots, x_{t+1}^{c_{t+1}}, \ldots, x_n^{a_n})+L,
%\]
%where $L$ is minimally generated by $\mu(K)-t$ monomials. Then
%\[
%d(I)=d(I^*)+t
%\quad\text{and}\quad
%A_i(I^*)\le A_i(I)-t
%\]
%for all $i$. Consequently $N(I^*)+t\le N(I)$, and the claim again follows by induction.
%\end{comment}

\begin{theorem}\label{Cmmon}
Let $S=k[x_1,\ldots,x_n]$ be a polynomial ring over a field $k$. If $I$ is a licci $\m$-primary monomial ideal, then $d(I)\le {\rm reg}(S/I).$
\end{theorem}

\begin{proof} By Proposition~\ref{TN} it suffices to prove that $N(I)\le {\rm reg}(S/I).$  We may assume that $b_1\ge 1$ in 
%the notation of 
Setting~\ref{S1}. Thus $I\subset  J=(x_1, x_2^{a_2},\ldots,x_n^{a_n})$ and  therefore $${\rm reg}(S/I)\ge {\rm reg}(S/J)=A_1(I)\ge N(I).$$
\end{proof}

\begin{example} We give an example (with the help of David Eisenbud) of a licci $\m$-primary ideal such that
the number of generators is exactly the highest twist at the end of the resolution.  This shows that Theorem \ref{Cmmon} is sharp.
Let $S = k[a,b,c]$, and let $$I: = (a^6,ab^5,b^6,a^5c,ab^4c,a^4c^2,a^3bc^2,a^2b^2c^2,ab^3c^2,a^3c^3,a^2c^4,ac^5,c^6).$$ The
algorithm of \cite{HU07} proves that $I$ is licci, it is equigenerated, and it is $\m$-primary of codimension three. It has 13 generators, and
the highest last twist of the resolution is also 13, as can be seen using Macaulay2 \cite{GS}. 

The licci condition is also certainly necessary; $\m^n$ gives an example in which the number of
generators exceeds the last twist whenever $n\geq 2$ and the number of variables is at least three. \end{example}

Corollary~\ref{Cmon} and Theorem~\ref{Cstrong} generalize Theorem~\ref{Cmmon} though they are,
in fact, consequences of this result. In the proof of Corollary~\ref{Cmon} we need to use the fact that
the licci property specializes:

\begin{theorem}\label{Tspec} Let $S$ be a local Gorenstein ring with infinite residue
field, $I$ be an ideal, and $R=S/(\underline y)$ for a sequence $\underline y$ that 
is regular on $S$ and $S/I.$
%$R$ be a ring obtained from $S$ by factoring out an ideal
%generated by a regular sequence on $S/I.$ 
If $I$ is licci, then so is $IR.$
\end{theorem}
\begin{proof} %First, notice that the sequence $\underline{y}$ is also regular on $S.$ 
Since $I$ is licci, a $t^{\rm th}$ universal link $L^t(I)$ is a complete intersection
for some $t \geq 0,$ see \cite{HU87}*{Theorem 2.17(b)}. The sequence $\underline y$ is regular
modulo each of the complete intersections defining the $t$ successive universal links and therefore, 
since $S/I$ is Cohen-Macaulay, $L^t(I)\otimes_SR$ is a universal link of $IR$ and a complete
intersection, see \cite{HU85}*{Lemma 2.12}. As the residue field $R$ is infinite, this implies 
that $IR$ is licci according to \cite{HU88}*{Theorem 2.4}.
\end{proof}

\begin{corollary}\label{Cmon}
Let $S=k[x_1,\ldots,x_n]$ be a polynomial ring over a field $k$ and $I$ be a licci monomial ideal of height $g.$ If $I$ contains a monomial regular sequence of length $g,$ then $\, d(I)\le {\rm reg}(S/I).$
\end{corollary}
\begin{proof} We may harmlessly assume that $k$ is infinite. Let $\alpha_1, \ldots, \alpha_g$ be a monomial regular sequence contained in
$I.$ Since the monomial supports of $\alpha_1, \ldots, \alpha_g$ are pairwise disjoint,
there is a surjective map of $k$-algebras $\pi: S \rightarrow R:=k[x_1, \ldots, x_g]$ sending the variables in the monomial support of each $\alpha_i$ to $x_i$ and
all the other variables to $0.$ Then $\pi(I)$ is a monomial ideal containing powers 
of the variables $x_1, \ldots, x_g.$ It follows that $\, {\rm ht} \, \pi(I)= {\rm ht} \, I$ and
so, since $S/I$ is Cohen-Macaulay, ${\rm ker \, \pi}$ is generated by a sequence of
homogeneous elements that are regular on $S$ and on $S/I.$ Therefore $d(\pi(I))=d(I)$ and,
since $S/I$ is Cohen-Macaulay,
$ {\rm reg}(R/\pi(I))={\rm reg}(S/I).$ Moreover, the ideal $\pi(I)_{(x_1, \ldots, x_g)}$ is licci by Theorem~\ref{Tspec}, and so $\pi(I)$ is licci by \cite{HU07}*{Theorem 2.6}.
%, as can be seen 
%by combining \cite[Theorem 2.17(b)]{HU87}, \cite[Proposition 2.8]{HU87}, and \cite[Theorem 2.4]{HU88}. 
Applying Corollary~\ref{Cmmon} to the zero-dimensional monomial ideal $\pi(I)$
now gives the assertion.
\end{proof} 

\begin{remark}\label{symbolic} The condition that a monomial ideal $I$ of height $g$ contains a
monomial regular sequence of length $g$ is satisfied
if the ideal is square-free and $I^{(g)} = I^g,$ i.e. the $g^{\rm th}$-symbolic power of $I$ is equal to $I^g$. One can see this as follows: consider the product
$m$ of all the variables in the ring. We claim that $m\in I^{(g)}.$ Since the ideal
$I$ is square-free, it is the intersection of prime ideals 
%$I$ is generated by square-free monomials, every associated prime of $I$ is minimal and is 
of the form $P = (x_{i_1},\ldots,x_{i_k}),$ where $k\geq g.$ Since $m$ is the
product of all the variables in the polynomial ring, clearly $m\in P^g$ for all
such $P$, and this forces $m\in I^{(g)}.$ Our assumption then gives that $m\in I^g$,
and hence we can write $m = m_1\cdots m_g,$ where each $m_i\in I.$ The supports of the
$m_i$ must be disjoint since $m$ is squarefree, so that $m_1,\ldots, m_g$ form a regular sequence
of monomials in $I$. This argument is well-known in the study of the so-called
packing problem, see \cite{GRV}.
\end{remark}

\begin{theorem}\label{Cstrong}
Let $S=k[x_1,\ldots,x_n]$ be a polynomial ring over a field $k$ and $I$ be a homogeneous  ideal of height $g$. If the initial ideal of $I$ with respect to some term order is licci and contains a monomial regular sequence of length $g,$ then $\, d(I)\le {\rm reg}(S/I).$
\end{theorem}
\begin{proof}Apply Corollary~\ref{Cmon} to
the initial ideal $\, {\rm in}(I)$ and use
the fact that $d(I) \leq d({\rm in}(I))$
and ${\rm reg}(S/I)={\rm reg}(S/{\rm in}(I));$ the last equality holds because $S/I$ and $S/{\rm in}(I)$ are Cohen-Macaulay, as can be seen from the Hilbert series.
\end{proof}

\section{Ideals of small Loewy colength}

In this section we will prove the following result:

\begin{theorem}\label{mcubed} Let $(R,\m)$ be a regular local ring containing a field and
$I$ be
an $\m$-primary licci ideal that contains $\m^s$.  If $s\leq 3$, then the deviation of $I$ is at most
$s-1$. 
\end{theorem}

In the course of the proof of this theorem, we will use the classification due to Poonen, of {\it all} zero-dimensional local algebras containing an algebraically closed residue field and having length at most $6$, up to
analytic isomorphism \cite{Poonen2008}. At the end of this section, we reproduce his table of
examples for the convenience of the reader.  We also calculated the additional
information we need on whether or not they are licci, their deviation, and their type.  

\begin{remark}\label{Poonen} Of critical importance for our proof is the observation that
for the local rings in Poonen's classification, the defining ideal is licci if and only if its deviation is at most $2.$
%if such an ideal (the defining ideal of the algebra) is licci then its deviation is at most $2$, 
Moreover the licci property implies that the type is at most $2$ except in one
case, when the algebra is $k[x,y]/(x,y)^3$, which has length $6$, deviation $2$, and
type $3$. 
\end{remark}

\noindent
{\em Proof of Theorem~\ref{mcubed}.}   If $s = 1$ the result is obvious so we assume that $s\geq 2$. We may complete $R$ without changing the fact that the extension of $I$ to the completion is licci and
contains $\m^s$. We may then choose a coefficient field and further assume the field is algebraically closed by making a faithfully flat extension. Again,
the extension of $I$ to this new ring is still licci and contains the same power of the new maximal ideal. Thus we may assume that
$R = k[[x_1,\ldots, x_n]]$, where $k$ is an algebraically closed field.

Assume first that $I$ is not contained in $\m^2.$ We make a change of variables to
write $I = (x_1, x_2,...,x_p, J)$ where $p: = \dim_k(\frac{I+\m^2}{\m^2}),$ and $J$ is
a finite colength ideal of the power series ring $k[[x_{p+1},\ldots, x_n]]$ that contains the $s^{\rm th}$ power of the maximal ideal of this smaller power series ring, and is contained in the square of the maximal ideal of this ring. By  Theorem \ref{Tspec} 
%\cite{HU07}*{Proposition 2.3}, $JR$ is licci in the ring $R$ as well, and then by 
%\cite{HU88}*{Theorem 2.12} it follows that 
the ideal $J$ is licci in $k[[x_{p+1},\ldots, x_n]].$ It suffices to prove that the deviation of $J$ is at most $s-1.$  By renaming our ideal and ring, we may assume that $I\subset \m^2$, and $R = k[[x_1,\ldots, x_n]].$

Assume that $s = 2.$ Then $I = \m^2.$ However, $\m^2$ is never licci if the
number of variables $n$ is at least $3,$ e.g. by the criterion in \cite{HU87}*{Corollary 5.13(a)}. Therefore, $n\leq 2,$ and in this case the deviation of $I$ is at most $1.$ 

We may now assume that $s = 3,$ but $I \neq \m^2.$ Since $\m^3\subset I\subset \m^2,$ we can write $I = (q_1,\ldots,q_t)+\m^3,$ where
$t: = \dim_k(\frac{I+\m^3}{\m^3}),$ and the $q_i$ may be assumed to be homogeneous
polynomials in $x_1, \ldots ,x_n.$  In particular, $I$ is extended from a homogeneous
ideal in the polynomial ring $S: = k[x_1, \ldots ,x_n].$ Call this ideal $J$, so that
$JR = I$. Let $T$ be the localization of $S$ at the maximal ideal $(x_1, \ldots ,x_n)S.$
Since the map from $T$ to $R$ is faithfully flat, \cite{HU88}*{Theorem 2.12} shows that
$JT$ is licci in $T$. We now use the criterion of \cite{HU87}: As a graded $S$-module, $S/J$ has a
minimal homogeneous free resolution, the minimal degree of a generator of $J$ is $2$, while the
maximal shift at the end of the resolution is $n+2$ since the socle of $S/J$
contains all homogeneous elements of degree $2.$ Since $JT$ is licci,  \cite{HU87}*{Corollary 5.13} shows that $$2(n-1) < n+2,$$ which implies that $n\leq 3$. We can assume that
$n = 3$ since the theorem is true if $n = 2$ by using the Hilbert-Burch theorem: any ideal containing the cube of the maximal ideal in a $2$-dimensional regular local ring has at most four generators and therefore has deviation at most $2$. For the remainder of the proof
we assume that $R = k[[x,y,z]]$.

%For the remainder of the proof we 
Now let $\lambda$ be the length of $R/I$,
$r$ the type of $I$, $\mu$ the minimal number of generators of $R/I$, and $d$ the deviation of $I$, namely $d = \mu - 3$.  Because $I$ is licci, Buchweitz \cite{Buc81} proved that the length of
the twisted conormal module, $\lambda(I/I^2\otimes_R \omega_{R/I}),$ is exactly
$3\lambda$, where $\omega_{R/I}$ is a canonical module of $R/I$. In particular,
\begin{equation}\label{conormal}\mu\cdot r\leq 3\lambda,
\end{equation}
since the number of generators of the twisted conormal
module is at most its length.

We proceed by examining the possible values of $t;$ recall that
$t: = \dim_k(\frac{I+\m^3}{\m^3}).$ As a side benefit of this analysis, we will prove that if $t\leq 2$ then $I$ cannot be licci.

Case $t = 0$.  In this case $I= \m^3,$ which, as we saw before, cannot be licci.
%$I = (x,y,z)^2$. However, it is well-known, e.g. using \cite[Corollary 5.13]{HU87} that this ideal is not licci.

Case $t = 1$. In this case there is a single quadric $q$ in $I$ up to scalar multiples,
and $I = (q,\m^3)$. We claim that $\lambda = 9$, $r\geq 5$, and $\mu = 8$.
The first calculation is clear.  The second follows since the entire square of the
maximal ideal is in the socle of $R/I$, and the dimension of the socle is the type.
The third follows since the span of $q$ in the space of cubics is $3$-dimensional, and $\m^3$
is minimally generated by $10$ cubics. Since $8\cdot 5$ is clearly not less than or equal to $3\cdot 9$,
(\ref{conormal}) proves that no such $I$ is licci. 

Case $t = 2$. We write $I = (q_1,q_2) + \m^3$.  If $q_1,q_2$ do not form a regular sequence,
then they have a common linear factor. In this case, the span of these elements in the
space of cubics has dimension
%$\m^3$ is 
at most $5$, since there is a linear relation among the two elements.
Therefore $\lambda = 8$, $r\geq 4$, and $\mu\geq 7$. Since $I$ is licci, $7\cdot 4\leq 3\cdot 8$ by (\ref{conormal}), a contradiction. Thus no such $I$ is licci.
It follows that we may assume $q_1$ and $q_2$ form a regular sequence, which we complete
to a maximal regular sequence $q_1,q_2,c$ in $I,$ where $c$ is a cubic. Let $L = (q_1,q_2,c).$
%generating the ideal $L = (q_1,q_2,c)$ where $c$ is a cubic. 
We claim that $L:I =\m^2$. Since $\m^2$ is not licci, this proves that
no $I$ with $t = 2$ is licci. To prove the claim, observe that $\m^5\subset L$, since
the socle of $R/L$ sits in degree $4= 1+1+2$. It follows that $\m^2I\subset L$ as $L$
already contains $q_1,q_2$. But the colength of $\, L:I \,$ is $\lambda(R/L) - \lambda(R/I) =
12-8 = 4 = \lambda(R/\m^2).$ So the claim follows.

Case $t = 3$. In this case $\lambda = 1+3+3 = 7.$ Write $I = (q_1,q_2,q_3) +\m^3$.  First suppose that the height of
the ideal $(q_1,q_2,q_3)$ is one.  Then the quadrics $q_1, q_2, q_3$ have a common linear factor, and there
are three linearly independent linear relations among them.  Moreover, the common linear factor will be in the socle, as will $\m^2/I$. Thus the type is $r\ge 4,$ and the minimal number of
generators satisfies $\mu\geq 3 + (10-6) = 7.$ 
%We also have that $\lambda = 1+3+3 = 7.$ 
Since $I$ is licci and $\lambda =7,$ we have $7\cdot 4\leq 3\cdot 7,$ a contradiction.  Thus
no such ideal is licci.  

We have shown that the height of the ideal generated by the quadrics in $I$ is at 
least $2$.
%If their height is three, then the quadrics form a regular sequence,
%and consequently have no linear relations.  It follows that the ideal $I$ is generated by these three quadrics together with one cubic.  This is an almost
%complete intersection which is licci, and the deviation is one.
%We have reduced to the case where the three quadrics generate an ideal of height two. We claim that $\mu =6.$ We apply \cite{CM}*{Theorem 2}, which proves the Eisenbud-Green conjecture in this case: there is a lex-plus powers ideal $(x^2,y^2,z^3) + J$ with the same Hilbert function as $I,$ where
%$J$ is a lex ideal.   Necessarily
%$J = (x^2,xy,y^2) + \m^3$, which has six generators. Thus, $I$ has exactly six generators. 
Choose a regular sequence consisting of forms of degrees $2,2,3$ contained in $I,$ and let
$L$ be the ideal these forms generate. The ideal $L:I$ is licci and has colength $\lambda(R/L) - \lambda(R/I) =
12-7 =5$. From Remark \ref{Poonen} it follows that the type of $R/(L:I)$ is at most 2. However, the type of this ring is the minimal number of generators of $I/L$, which in turn is at least $d(I).$ %which is at least three, a contradiction. Thus no such ideal is licci.
Thus $d(I)\le 2.$

Case $t \geq 4.$  In this case, 
%the length of $R/I$ is at most $6$. However, if 
$\lambda(R/I)\leq 6.$ So Remark \ref{Poonen} shows that the deviation of $I$ is at most
$2$, proving the theorem in this case. 
\qed

\vspace{.3cm}

To finish, we give Poonen's table
\footnote{Three of Poonen's examples of colength 6 have a redundant generator $z^4,$ which we have removed in our table. A $^*$ indicates an example in characteristic $2$ and a $^{**}$ indicates an example in characteristic $3$. The table we use is from a corrected version of the original paper; there were a few
changes Poonen made in some of the characteristic $2$ examples.} with the added information we need; here
$R$ is a power series ring over an algebraically closed field in the
variables $x, y, \ldots :$

\begin{longtable}{@{}>{\small $}c<{$} c c c c@{}}
\caption{Ideals appearing in the classification.}
\label{tab:ideals} \\
%\toprule
%\multicolumn{1}{c}{$I$} & Licci & $d(I)$ & $r(I)$ \\
%\midrule
%\endfirsthead

%\toprule
\multicolumn{1}{c}{$I$} & $\lambda(R/I)$ & Licci & $d(I)$ & $r(I)$ \\
\toprule
%\midrule
%\endfirsthead

%\bottomrule
%\endfoot

(0) & 0& YES & 0& 1 \\
(x^2) & 2& YES & 0& 1\\
(x^3) &3 & YES& 0& 1\\
(x,y)^2 &3& YES& 1& 2\\
(x^4) &4& YES& 0& 1\\
(x^2,xy,y^3) &4&YES & 1& 2\\
(x^2,y^2) &4& YES& 0& 1\\
^*(x^2+y^2,xy) &4&YES & 0&1 \\
(x,y,z)^2 &4& NO & 3& 3\\
(x^5) & 5 &YES& 0& 1\\
(x^2,xy,y^4) &5& YES& 1& 2\\
(x^2+y^3,xy) &5& YES& 0& 1\\
(xy,x^3,y^3) &5& YES& 1& 2\\
(x^2,xy^2,y^3) &5& YES& 1& 2\\
(x^2,y^2,xy,xz,yz,z^3) &5&NO & 3& 3\\
(x^2,y^2,z^2,xy,xz) &5& YES& 2& 2\\
^*(x^2,xy,xz,yz,y^2+z^2) &5&YES & 2& 2 \\
(xy,xz,yz,x^2+y^2,x^2+z^2) &5&YES & 2 & 1\\
(x,y,z,w)^2 &5& NO & 6& 4\\
(x^6) &6& YES& 0& 1\\
(x^2,xy,y^5) &6& YES& 1& 2\\
(x^2+y^4,xy) &6& YES& 0& 1\\
(xy,x^3,y^4) &6&YES & 1& 2\\
(xy,x^3+y^3) &6& YES& 0& 1\\
(x^2,xy^2,y^4) &6& YES& 1& 2\\
(x^2+y^3,xy^2,y^4) &6& YES& 1& 2\\
(x^2,y^3) &6& YES& 0& 1\\
^*(x^2+xy^2,y^3) &6&YES & 0& 1\\
^{**}(x^2,xy^2+y^3) &6& YES&0 & 1\\
(x,y)^3 &6&YES & 2& 3\\
(x^2,xy,y^2,xz,yz,z^4) &6& NO &3 &3 \\
(x^2,xy,y^2+z^3,xz,yz) &6& YES& 2& 1\\
(x^2,xy+z^3,y^2,xz,yz) &6& YES& 2& 1\\
^*(x^2+z^3,xy,y^2+z^3,xz,yz) &6& YES& 2& 1\\
(xy,yz,z^2,y^2-xz,x^3) &6& YES& 2& 2\\
(xy,z^2,xz-yz,x^2+y^2-xz) &6& YES& 1& 2\\
^*(x^2,z^2,y^2-xz,yz) &6& YES& 1& 2\\
(x^2,xy,xz,y^2,yz^2,z^3) &6& NO& 3& 3\\
(x^2,xy,xz,yz,y^3,z^3) &6& NO& 3& 3\\
(xy,xz,y^2,z^2,x^3) &6& YES& 2& 2\\
^*(xy,xz,yz,y^2-z^2,x^3) &6 &YES & 2& 2\\
(xy,xz,yz,x^2+y^2-z^2) &6& YES& 1& 2\\
(x^2,xy,yz,y^2-z^2) &6& YES & 1& 2\\
^*(x^2,xy,yz,xz+y^2-z^2) &6& YES& 1& 2\\
(x^2,xy,y^2,z^2) &6&YES &1 & 2\\
^*(x^2,xy,y^2,z^2-xz) &6& YES& 1& 2\\
(x^2,y^2,z^2,xy,xz,xw,yz,yw,zw,w^3) &6&NO & 6& 4\\
(x^2,y^2,z^2,w^2,xy,xz,xw,yz,yw) &6& NO& 5& 3 \\
^*(x^2,y^2,z^2+w^2,xy,xz,xw,yz,yw,zw) &6& NO& 5& 3\\
(x^2,y^2+z^2,y^2+w^2,xy,xz,xw,yz,yw,zw)&6 & NO& 5& 1\\
(x^2,y^2,z^2,w^2,xy-zw,xz,xw,yz,yw) &6& NO & 5& 1\\
^*(x^2+y^2,x^2+z^2,x^2+w^2,xy,xz,xw,yz,yw,zw) &6& NO & 5& 1\\
(x,y,z,w,v)^2 &6& NO& 10& 5\\

\end{longtable}

\bibliography{references}

\end{document}